\documentclass[11pt, reqno,sumlimits]{article}
\usepackage{amssymb,amscd, epsfig, mathrsfs, xypic, amsmath,amsthm, color} 
\usepackage{hyperref}
\usepackage{blindtext}
\usepackage{enumerate}
\usepackage{ascmac} 
\usepackage{graphicx}

\xyoption{all}
\newtheorem{thm}{Theorem}[section]  
\newtheorem{cor}[thm]{Corollary}
\newtheorem{lem}[thm]{Lemma} 
 
\newtheorem{df-pr}[thm]{Definition-Proposition}

\theoremstyle{definition}
\newtheorem{defn}[thm]{Definition}
 
\newtheorem{rem}[thm]{Remark}

\newcommand{\CC}{{\mathbb C}}

\newcommand{\ZZ}{{\mathbb Z}}

\newcommand{\PP}{{\mathbb P}}

\newcommand{\sfm }{{\mathsf m }}

\newcommand{\sfs }{{\mathsf s}}

\newcommand{\calL}{{\mathcal L}}

\newcommand{\calO}{{\mathcal O}}
\newcommand{\calP}{{\mathcal P}}
\newcommand{\calQ}{{\mathcal Q}}

\newcommand{\calS}{{\mathcal S}}

\newcommand{\scL}{{\mathscr L}}

\newcommand{\scP}{{\mathscr P}}

\newcommand{\scS}{{\mathscr S}}

\newcommand{\SP}{\calS\calP}

\newcommand{\lan}{{\langle}}
\newcommand{\ran}{{\rangle}}

\newcommand{\inc}{\hookrightarrow}

\newcommand{\Fl}{\operatorname{Fl}}

\newcommand{\pt}{\operatorname{pt}}

\newcommand{\gr}{\operatorname{gr}}

\newcommand{\rk}{{\operatorname{rk}}}

\newcommand{\supp}{{\operatorname{supp}}}

\newcommand{\Pf}{{\operatorname{Pf}}}
\newcommand{\OG}{{\operatorname{OG}}}
\newcommand{\GP}{{\textit{GP}}}

\newcommand{\CK}{{\textit{CK}}}
\newcommand{\CH}{{\textit{CH}}}

\newsavebox{\savepar}

\numberwithin{equation}{section}

\newcounter{labelflag} 
\newcommand{\labelon}{\setcounter{labelflag}{1}}
\newcommand{\Label}[1]{\ifnum\thelabelflag=1\ifmmode
\makebox[0in][l]{\qquad\fbox{\rm#1}} \else
\marginpar{\vspace{0.7\baselineskip} \hspace{-1.1\textwidth}
\fbox{\rm#1}} \fi \fi \label{#1} } \labelon
\begin{document} 
\title{Pfaffian formula for $K$-theory of odd orthogonal Grassmannians}
\author{Thomas Hudson, Takeshi Ikeda, Tomoo Matsumura, Hiroshi Naruse}
\date{}
\maketitle 
\begin{abstract}
We prove a Pfaffain formula for the $K$-theory class of the degeneracy loci in the bundle of odd maximal orthogonal Grassmannians.
\end{abstract}
%\tableofcontents
%%%%%%%%%%%%%%%%%%%%%%%%%%%%%%%%%%
%%%%%%%%%%%%%%%%%%%%%%%%%%%%%%%%%%
\section{Introduction}
%%%%%%%%%%%%%%%%%%%%%%%%%%%%%%%%%%
%The determinantal formula of Schubert calculus goes back to Giambelli \cite{Giambelli}. It expresses an arbitrary Schubert class in the Chow ring of the Grassmannian variety as a determinant that is similar to the Jacobi-Trudi formula for a Schur function.  In 1974, Kempf and Laksov \cite{KempfLaksov} proved an extension of Giambelli's formula to the degeneracy loci in the  Grassmann bundle associated to a vector bundle. Their formula expresses the class of a degeneracy locus as a determinant whose entries are the Chern classes of the vector bundle and the tautological vector bundles. Analogous cohomological formulas for a vector bundle with a non-degenerate quadratic form have been also extensively studied (\cite{AndersonFulton2}, \cite{BuchKreschTamvakis2}, \cite{BuchKreschTamvakis1}, \cite{Kazarian}, \cite{KreschTamvakis2002}, \cite{LascouxPragacz1998}, \cite{LascouxPragacz2000}, \cite{PragaczRatajski1997}, \cite{Tamvakis13}, \cite{TamvakisJAlgGeom}). 
%Most notable here is Kazarian's work \cite{Kazarian}. He proved a Pfaffian formula for the degeneracy loci of the both Lagrangian and maximal orthogonal subbundles.  In \cite{HIMN}, authors generalized his method to the $K$-theory, and proved a determinantal formula for the Grassmann bundle and also Pfaffian formula for the Lagrangian and non-maximal isotropic Grassmann bundles of a vector bundle with a symplectic form.

 This paper is a sequel to our recent study \cite{HIMN} on the $K$-theoretic degeneracy loci formulas for the vector bundles. In \cite{HIMN}, we proved a determinant formula for Grassmann bundles, and a Pfaffian formula for isotropic Grassmann bundles associated to a symplectic vector bundle. Both formulas live in the $K$-ring of algebraic vector bundles. In this paper, we deal with a vector bundle of odd rank equipped with a non-degenerate symmetric bilinear form. 

Our study is modeled on Kazarian's work \cite{Kazarian} for the Lagrangian and orthogonal degeneracy loci in cohomology. We succeeded in \cite{HIMN} to extend Kazarian's approach to $K$-theory for both ordinary and symplectic vector bundles. The main technical issue in the orthogonal case is that one has to deal with schemes that are not reduced and in particular this means that it is necessary to find a way to recover the fundamental class of the associated reduced scheme. In cohomology, this is achieved by dividing by $2$, so the resulting cohomology classes of Lagrangian and maximal orthogonal degeneracy loci differ only by a multiple of a power of $2$. This reflects the difference between the root systems of type $B$ and $C$. %It turns out that this difference is more subtle but significant in the computation of the $K$-theory classes of the corresponding degeneracy loci. 
We reduce this issue related to the multiplicity to the case of the quadric bundle and deal with it in the appendix.

We can also see this subtlety in the fact that the structure constants for the natural basis formed by the structure sheaves of the Schubert varieties are quite different for Lagrangian and the maximal orthogonal Grassmannian. Such structure constants are known explicitly for the maximal orthogonal Grassmannian by the work of Pechenik and Yong \cite{YongPechenik}, while for the Lagrangian case no conjecture has been proposed.

In the Lagrangian case, the special classes given by the degeneracy loci with only one Schubert condition coincide with the Segre classes of tautological vector bundles. The general degeneracy loci classes are given by a Pfaffian with entries being a quadratic expression in terms of those Segre classes. On the other hand, in the orthogonal case, the special class is given as a series in terms of the Segre classes of the tautological bundles (see Lemma \ref{special}). We introduce classes $\scP_{m}^{(\ell)}$ that can be considered as a deformation of the special ones (see Definition \ref{defPclass}). Our main theorem (Theorem \ref{mainthm}) describes an arbitrary degeneracy loci class as a Pfaffian with entries being a quadratic expression in terms of those classes $\scP_m^{(\ell)}$. If we specialize it to the $K$-theory of orthogonal Grassmannian, it is a Pfaffian formula of Schubert classes given in terms of honest special classes. This is in the spirit of Giambelli \cite{Giambelli}.

Our formula for the orthogonal degeneracy loci does not change the form when we consider an orthogonal Grassmann bundle for higher rank orthogonal vector bundle. Thus we can think of the degeneracy loci class in infinite rank setting. Such universal class should be described by the $\GP$-functions defined by the second and the forth authors in \cite{IkedaNaruse}. The result of this paper implies that those $\GP$-functions should be expressed as a Pfaffian. Further combinatorial implications of our result will be discussed elsewhere.

%Our formula is stable with respect to the natural inclusions of an orthogonal Grassmannian into the one associated with a vector bundle with larger rank. Therefore it defines a family of universal ``polynomials'' that represent Schubert classes. Ikeda-Naruse \cite{IkedaNaruse} defined $\GP$ functions that represent equivariant $K$-theory Schubert classes. Therefore our formula gives a Pfaffian formula for those $\GP$ functions. This aspect of the theory will be studied elsewhere.

We expect that our formula can be generalized to the degeneracy loci classes associated to the \emph{vexillary signed permutations} due to Anderson-Fulton. In fact, in \cite{AndersonFulton2} they obtained a Pfaffian formula in cohomology, which we plan to achieve at the level of $K$-theory, using our method.

%%%%%%%%%%%%%%%%%%%%%%%%%%%%%%%%%%
\section{Basics on connective $K$-theory}
%%%%%%%%%%%%%%%%%%%%%%%%%%%%%%%%%%
Connective $K$-theory, denoted by $\CK^*$, is an example of oriented cohomology theory built out of the algebraic cobordism introduced by Levine and Morel \cite{LevineMorel}. It is a contravariant functor, together with pushforwards for projective morphisms, satisfying some axioms. We refer to \cite{DaiLevine, Hudson, LevineMorel} for the detailed construction. In this section, we recall some preliminary facts on $\CK^*$, especially regarding Chern classes.

%%%%%%%%%%%%%%%%%%%%%%%%%%%%%%%%%%
Let $X$ be a smooth quasiprojective variety over the complex numbers $\CC$. The connective $K$-theory of $X$ interpolates between the Grothendieck ring $K(X)$ of algebraic vector bundles on $X$ and the Chow ring $\CH^*(X)$ of $X$. Connective $K$-theory assigns to $X$  a commutative graded algebra $\CK^*(X)$  over the  coefficient ring $\CK^*(\pt)$ which is isomorphic to the polynomial ring $\ZZ[\beta]$ by setting $\beta$ to be the class of degree $-1$ obtained by pushing forward the fundamental class along the structural morphism $\PP^1 \to \pt$. The $\ZZ[\beta]$-algebra $\CK^*(X)$ specializes to the Chow ring $\CH^*(X)$ and the Grothendieck ring $K(X)$ by setting $\beta$ equal to $0$ and $-1$ respectively.  For any closed equidimensional subvariety $Y$ of $X$, there exists an associated fundamental class  $[Y]_{\CK^*}$ in $\CK^*(X)$. In particular,  $[Y]_{\CK^*}$ is specialized to the class $[Y]$ in $\CH^*(X)$ and also to the class of the structure sheaf $\mathcal{O}_Y$ of $Y$ in $K(X)$. In the rest of the paper, we denote the fundamental class of $Y$ in $\CK^*(X)$ by $[Y]$ instead of $[Y]_{\CK^*}$.

As an oriented cohomology theory, connective $K$-theory admits A theory of Chern classes. For line bundles $L_1$ and $L_2$ over $X$, one has first Chern classes $c_1(L_i)\in  \CK^1(X)$ which satisfy
\begin{equation}\label{L tensor L}
c_1(L_1\otimes L_2)=c_1(L_1)+c_1(L_2)+\beta c_1(L_1)c_1(L_2).
\end{equation}
This fundamental law characterizes the theory. Note that the operation
\[
(u,v)\mapsto u\oplus v:= u+v+\beta uv
\]
is an example of commutative one-dimensional formal group law, which is an essential feature of oriented cohomology theories. We should stress here that the sign convention of $\beta$ is opposite from the one used in \cite{DaiLevine, Hudson, LevineMorel}.  

On the other hand, it follows from (\ref{L tensor L}) that $c_1(L^{\vee}) = -c_1(L)/(1+\beta c_1(L))$. Therefore it is convenient to introduce the notation for the \emph{formal inverse}: 
\[
\bar u := \frac{-u}{1+\beta u}.
\]

%%%%%%%%%%%%%%%%%%%%%%%%%%%%%%%%%%

%%%%%%%%%%%%%%%%%%%%%%%%%%%%%%%%%
%%%%%%%%%%%%%%%%%%%%%%%%%%%%%%%%%%
%%%%%%%%%%%%%%%%%%%%%%%%%%%%%%%%%%
%\subsection{Segre classes}
%%%%%%%%%%%%%%%%%%%%%%%%%%%%%%%%%%
%%%%%%%%%%%%%%%%%%%%%%%%%%%%%%%%%%
The main ingredient in our computation is the $K$-theoretic Segre class of vector bundles. Let $E$ be a vector bundle over $X$ of rank $e$. For convenience, we use the following notation to denote the Chern polynomial:
\[
c(E;u):=\sum_{i=0}^{e} c_i(E) u^i.
\]
Let $F$ be another vector bundle over $X$. In \cite{HIMN}, we defined the relative Segre class $\scS_m(E-F)$ for each $m\in \ZZ$ by using the following generating function:
\begin{equation}\label{segre vir}
\scS(E-F;u):=\sum_{m\in \ZZ} \scS_{m}(E-F) u^{m}= \frac{1}{1 + \beta u^{-1}} \frac{c(E - F;\beta)}{c(E-F;-u)},
\end{equation}
where $c(E-F;u):=c(E;u)/c(F;u)$ defines the usual relative Chern classes. It was shown in \cite{HIMN} that $\scS_m(E-F)$ can be also obtained as the pushforward of the product of certain Chern classes as follows.
\begin{lem}\label{lemtensor}
Let $\pi: \PP^*(E)\to X$ be the dual projective bundle of $E$ and $\calQ$ its tautological quotient line bundle. For each integer $s\geq 0$, we have
\begin{equation}\label{push of tensor} 
\pi_*\left(c_1(\calQ)^sc_{f}(\calQ \otimes F^{\vee})\right) = \scS_{s+f-{e}+1}(E-F),
\end{equation}
where $f$ is the rank of $F$.
\end{lem}
%%%%%%%%%%%%%%%%%%%%%%%%%%%%%%%%%%
%%%%%%%%%%%%%%%%%%%%%%%%%%%%%%%%%%
%%%%%%%%%%%%%%%%%%%%%%%%%%%%%%%%%%
%%%%%%%%%%%%%%%%%%%%%%%%%%%%%%%%%%
\section{Maximal Orthogonal Grassmannians of type $B$}
In this section, we first define the degeneracy loci in the odd orthogonal Grassmann bundle. In order to compute its associated class, we construct a resolution of singularities. With the help of Lemma \ref{cor1} on the quadric bundle (proved in the appendix), we express the degeneracy loci class as a pushforward of a product of top Chern classes. With the help of the calculus of formal Laurent series developed in \cite{HIMN}, we finally obtain the Pfaffian formula (Theorem \ref{mainthm}).
%%%%%%%%%%%%%%%%%%%%%%%%%%%%%%%%%%
%%%%%%%%%%%%%%%%%%%%%%%%%%%%%%%%%%
\subsection{Degeneracy loci}\label{secKL}
%%%%%%%%%%%%%%%%%%%%%%%%%%%%%%%%%%
Let  $X$ be a smooth quasiprojective variety. Consider the vector bundle $E$ of rank $2n+1$ over $X$ with a symmetric non-degenerate bilinear form $\lan \ ,\  \ran: E \otimes E \to \calO$ where $\calO$ is the trivial line bundle. Let $\xi: \OG(E) \to X$ be the Grassmann bundle parametrizing rank $n$ isotropic subbundles of $E$, equipped with the tautological bundle $U$. A point of $\OG(E)$ is a pair $(x,U_x)$ of a point $x\in X$ and an isotropic $n$-dimensional subspace of the fiber $E_x$ of $E$ at $x$. Fix a flag of isotropic subbundles of $E$
\[
F^n \subset \cdots \subset F^2 \subset F^1,
\]
where $\rk\  F^i = n-i+1$. Let $F^{-i+1}:=(F^{i})^{\perp}$. Note that the bilinear form $\lan\ ,\ \ran$ on $E$ induces an isomophism  $F^{\perp}/F \otimes F^{\perp}/F \cong \calO$ for any maximal isotropic subbundle $F$ of $E$. This implies that $c_1(F^{\perp}/F)=0$ in  $\CK^*(X)\otimes_{\ZZ}\ZZ[1/2]$.

A strict partition of at most $n$ parts is a sequence $\lambda=(\lambda_1,\dots, \lambda_n)$ of non-negative integers such that $\lambda_i>0$ implies $\lambda_i>\lambda_{i+1}$ for all $i=1,\dots n-1$. Let $\SP(n)$ be the set of such strict partitions $\lambda$ such that $\lambda_1\leq n$. The length of $\lambda$ is the number of nonzero parts. For each partition $\lambda \in \SP(n)$ of length $r$, the corresponding degeneracy loci $X_{\lambda}$ in $\OG(E)$ is defined by 
\[
X_{\lambda} = \{ (x,U_x) \in \OG(E)\ |\ \dim(U_x \cap F^{\lambda_i}_x)\geq i, i=1,\dots,r\}.
\]
%%%%%%%%%%%%%%%%%%%%%%%%%%%%%%%%%%
%%%%%%%%%%%%%%%%%%%%%%%%%%%%%%%%%%
\subsection{Resolution of singularities}
%%%%%%%%%%%%%%%%%%%%%%%%%%%%%%%%%%
%%%%%%%%%%%%%%%%%%%%%%%%%%%%%%%%%%
Let $\pi: \Fl(F_{\bullet}^{\lambda}) \to \OG(E)$ be the flag bundle associated to $F_{\bullet}^{\lambda}: F^{\lambda_1} \subset \cdots \subset F^{\lambda_r}$ with the tautological flag $D_1\subset \cdots \subset D_r$ with $\rk\  D_i=i$. That is, for each point $p:=(x,U_x) \in \OG(E)$, its fiber along $\pi$ consists of the partial flag $(D_1)_p\subset \cdots \subset (D_r)_p$ of $E_x$ such that $\dim (D_i)_p=i$ and $(D_i)_p \subset F^{\lambda_i}_p$. We can construct the associated flag bundle $\pi: \Fl(F_{\bullet}^{\lambda})  \to \OG(E)$ as a tower of projective bundles
\begin{eqnarray}
&&\Fl(F_{\lambda}^{\bullet})=\PP(F^{\lambda_r}/D_{r-1}) \stackrel{\pi_r}{\longrightarrow} 
\PP(F^{\lambda_{r-1}}/D_{r-2}) \stackrel{\pi_{r-1}}{\longrightarrow} \cdots \ \ \ \ \ \ \ \ \ \nonumber\\\label{P tower}
&&\ \ \ \ \ \ \ \ \ \ \ \ \ \ \ \ \ \ \ \ \ \ \cdots\stackrel{\pi_3}{\longrightarrow} \PP(F^{\lambda_2}/D_1) \stackrel{\pi_2}{\longrightarrow} \PP(F^{\lambda_1})  \stackrel{\pi_1}{\longrightarrow} \OG(E).
\end{eqnarray}
We regard $D_i/D_{i-1}$ as the tautological line bundle of $\PP(F^{\lambda_i}/D_{i-1})$ and denote by $\tau_i$ the first Chern class $c_1((D_i/D_{i-1})^{\vee})$ in $\CK^*(\PP(F^{\lambda_i}/D_{i-1}))$.

%%%%%%%%%%%%%%%%%%%%%%%%%%%%%%%%%%
Define the sequence of subvarieties $Y_r\subset \cdots \subset Y_1 \subset \Fl(F_{\bullet}^{\lambda})$ by
\[
Y_i = \{ (p, (D_{\bullet})_p) \in \Fl(F_{\bullet}^{\lambda}) \ |\ p=(x,U_x), \ (D_i)_p \subset U_x\}.
\]
%%%%%%%%%%%%%%%%%%%%%%%%%%%%%%%%%%
\begin{lem}
$Y_i$ is smooth and $Y_r$ is birational to $X_{\lambda}$ through $\pi$. Furthermore, $\pi_*[Y_r] = [X_{\lambda}]$.
\end{lem}
\begin{proof}
To prove that $Y_i$ is smooth, it suffices to prove it for $i=r$. Since $F^i$'s are vector bundles over $X$, we have the flag bundle $\Fl(F_{\bullet}^{\lambda})' \to X$ associated to the partial flag $F_{\bullet}^{\lambda}$ defined as above.  It is easy to see that $\Fl(F_{\bullet}^{\lambda})=\Fl(F_{\bullet}^{\lambda})'\times \OG(E)$. Let $\xi_1: \Fl(F_{\bullet}^{\lambda}) \to \Fl(F_{\bullet}^{\lambda})'$ be the projection to its first factor. Then $Y_r$ surjects to $\Fl(F_{\bullet}^{\lambda})'$ and each fiber is $\OG((D_r)_p^{\perp}/(D_r)_p)$. Thus we can see that $Y_r$ is a fiber bundle over $\Fl(F_{\bullet}^{\lambda})'$ with the fiber being identified with the maximal orthogonal Grassmannian $\OG(\CC^{2(n-r)+1})$ of $\CC^{2(n-r)+1}$. Thus it is smooth. The birationality is clear. The last claim follows from the fact that $X_{\lambda}$ has at worst rational singularities (\textit{cf.} \cite[Lemma 4]{HIMN}). 
\end{proof}
%%%%%%%%%%%%%%%%%%%%%%%%%%%%%%%%%%
\begin{lem}\label{lemY_r}
Let %$\tau_i:=c_1((D_i/D_{i-1})^{\vee})$ and 
$\kappa:=c_1(U^{\perp}/U)$. In $\CK^*(\Fl(F^{\lambda}_{\bullet}))$, we have
\[
\left(\prod_{i=1}^r (2 + \beta(\tau_i\oplus \kappa) )\right)[Y_r] = \prod_{i=1}^r c_{n-i+1}((D_i/D_{i-1})^{\vee}\otimes D_{i-1}^{\perp}/U^{\perp}).
\]
\end{lem}
%%%%%%%%%%%%%%%%%%%%%%%%%%%%%%%%%%
\begin{proof}
Consider the vector bundle $D_{i-1}^{\perp}/D_{i-1}$ over $Y_{i-1}$ with the induced bilinear form. Let $Q(D_{i-1}^{\perp}/D_{i-1})$ be the corresponding quadric bundle over $Y_{i-1}$ (see Section \ref{appendix}). Let $S_i$ be the tautological line bundle of $Q(D_{i-1}^{\perp}/D_{i-1})$ and consider the maximal isotropic subbundle $U/D_{i-1}$ of $D_{i-1}^{\perp}/D_{i-1}$. We apply Lemma \ref{cor1} and obtain the following identity in $\CK^*(Q(D_{i-1}^{\perp}/D_{i-1}))$:
\begin{equation}\label{eqi-1}
(2+\beta c _1(S_i^{\vee}\otimes U^{\perp}/ U)) [\PP(U/D_{i-1})] = c_{n-i+1}(S_i^{\vee}\otimes D_{i-1}^{\perp}/U^{\perp}).
\end{equation}
The line bundle $D_i/D_{i-1} \to Y_{i-1}$ defines a section $a_i: Y_{i-1} \to Q(D_{i-1}^{\perp}/D_{i-1})$ by sending a point $y\in Y_{i-1}$ to its fiber of $D_i/D_{i-1}$ which is isotropic since $D_i \subset F^{\lambda_i}$.  The pullback of $S_i$ along $a_i$ coincides with $D_i/D_{i-1}$. Furthermore $a_i^*[\PP(U/D_{i-1})]=[Y_i]$. Thus by pulling back (\ref{eqi-1}) along $a_i$, we obtain the following identity in $\CK^*(Y_{i-1})$:
\begin{eqnarray*}
(2+\beta (\tau_i\oplus \kappa))[Y_i] &=& c_{n-i+1}((D_i/D_{i-1})^{\vee}\otimes D_{i-1}^{\perp}/U^{\perp}).
\end{eqnarray*}
The claim follows by the projection formula applied to the inclusions $Y_{i-1} \inc Y_i$.
\end{proof}
%%%%%%%%%%%%%%%%%%%%%%%%%%%%%%%%%%
The next corollary  is an obvious consequence of Lemma \ref{lemY_r} and the fact that $\kappa=c_1(U^{\perp}/U)=0$ in $\CK^*(\OG(E))\otimes_{\ZZ}\ZZ[1/2]$.
%%%%%%%%%%%%%%%%%%%%%%%%%%%%%%%%%%
\begin{cor}\label{corXlambda}
In $\CK^*(\OG(E))\otimes_{\ZZ}\ZZ[1/2]$, we have
\[
[X_{\lambda}] = \pi_* \left(\prod_{i=1}^r\frac{ c_{n-i+1}((D_i/D_{i-1})^{\vee}\otimes D_{i-1}^{\perp}/U^{\perp})}{2 + \beta\tau_i}\right).
\]
\end{cor}
\emph{In the rest of the paper, we work in $\CK^*(\OG(E))\otimes_{\ZZ}\ZZ[1/2]$.}
%%%%%%%%%%%%%%%%%%%%%%%%%%%%%%%%%%
\subsection{Special Schubert classes}
%%%%%%%%%%%%%%%%%%%%%%%%%%%%%%%%%%
Let us now  introduce the classes $\scP_m^{(\ell)}$, which will be used in the main theorem to describe the fundamental class of a general degeneracy loci.
%%%%%%%%%%%%%%%%%%%%%%%%%%%%%%%%%%
\begin{defn}\label{defPclass}
For each $m\in \ZZ$ and $\ell=1,\dots, n$, we define the classes $\scP_m^{(\ell)}$ by the following generating function
\begin{equation*}
\sum_{m\in \ZZ}\scP_m^{(\ell)} u^m = \frac{1}{2+ \beta u^{-1}} \scS((U^{\perp}- E/F^{\ell})^{\vee};u).
\end{equation*}
Or equivalently, we define
\begin{equation}\label{eqP}
\scP_m^{(\ell)}=\frac{1}{2} \sum_{s\geq 0} \left(-\frac{\beta}{2}\right)^s\scS_{s+m}((U^{\perp}- E/F^{\ell})^{\vee}).
\end{equation}
\end{defn}
%%%%%%%%%%%%%%%%%%%%%%%%%%%%%%%%%%
For each integer $k=1,\dots,n$, let $\lambda=(k) \in \SP(n)$ be the strict partition with only one part. The corresponding degeneracy loci is denoted by $X_k$ and its associated class $[X_k]$ is called a \emph{special class}. 
The next lemma shows that we can view $\scP_k^{(\ell)}$ as a deformation of the special class $[X_k]$. 
%%%%%%%%%%%%%%%%%%%%%%%%%%%%%%%%%%
\begin{lem}\label{special} 
We have $[X_{k}]=\scP_{k}^{(k)}$.
\end{lem}
%%%%%%%%%%%%%%%%%%%%%%%%%%%%%%%%%%
\begin{proof}
As in Section \ref{secKL}, we consider $\pi: \PP(F^{k}) \to \OG(E)$ and $Y_1 \subset \PP(F^{k})$ which is the loci where $D_1$ is contained in $U$. By Lemma \ref{lemY_r}, we have
\[
(2 + \beta(\tau_1\oplus \kappa) [Y_1] = c_{n}(D_1^{\vee}\otimes E/U^{\perp}).
\]
As in Corollary \ref{corXlambda}, we get
\[
[X_{k}] = \pi_* \left(\frac{ c_{n}(D_1^{\vee}\otimes E/U^{\perp})}{2 + \beta\tau_1}\right)=
\frac{1}{2} \sum_{s\geq 0} \left(-\frac{\beta}{2}\right)^s \pi_*\left(\tau_1^sc_n(D_1^{\vee}\otimes E/U^{\perp})\right)
\]
in $\CK^*(\OG(E))\otimes_{\ZZ}\ZZ[1/2]$. Now the claim follows from Lemma \ref{lemtensor}.
\end{proof}
%%%%%%%%%%%%%%%%%%%%%%%%%%%%%%%%%%

%We conclude this section by noting that the generating function of the classes $\scP_m^{(\ell)}$ can be computed in terms of the generating function of Segre classes.
%%%%%%%%%%%%%%%%%%%%%%%%%%%%%%%%%%%
%\begin{lem}\label{defPclass}
%\[
%\sum_{m\in \ZZ}\scP_m^{(\ell)} u^m = \frac{1}{2+ \beta u^{-1}} \scS((U^{\perp}- E/F^{\ell})^{\vee};u) 
%\]
%\end{lem}
%\begin{proof}
%It follows from a direct computation. Indeed, we have
%\begin{eqnarray*}
%\sum_{m\in \ZZ} \scP_m^{(\ell)} u^m 
%&=& \sum_{m\in \ZZ}  \frac{1}{2} \sum_{s\geq 0} \left(-\frac{\beta}{2}\right)^s\scS_{s+m}((U^{\perp}- E/F^{\ell})^{\vee} )u^m\\
%&=& \frac{1}{2} \sum_{s\geq 0} \left(-\frac{\beta}{2}\right)^s\frac{1}{u^{s}}\sum_{m\in \ZZ}  \scS_{s+m}((U^{\perp})^{\vee}- E/F^{\ell} )u^{m+s}.
%\end{eqnarray*}
%\end{proof}
%%%%%%%%%%%%%%%%%%%%%%%%%%%%%%%%%%
%%%%%%%%%%%%%%%%%%%%%%%%%%%%%%%%%%
\subsection{Computing $[X_{\lambda}]$}\label{secPFthm}
%%%%%%%%%%%%%%%%%%%%%%%%%%%%%%%%%%
First, we recall notations from \cite{HIMN} for the ring of certain formal Laurent series, necessary for the computation of the class $[X_{\lambda}]$. Let $R=\oplus_{m\in \ZZ}R_m$ be a commutative $\ZZ$-graded ring. Let  $t_1,\ldots,t_r$ be indeterminates with $\deg(t_i)=1$. For ${\sfs}=(s_1,\ldots,s_r)\in \ZZ^r$, we denote $t^{\sfs}=t_1^{s_1}\cdots t_r^{s_r}$. A formal Laurent series of degree $m$ in the variables $t_1,\ldots,t_r$ with coefficients in $R$ is given by
\[
f(t)=\sum_{\sfs\in \ZZ^r}a_{\sfs}t^{\sfs},
\]
where $a_{\sfs}\in R_{m-|{\sfs}|}$ for all ${\sfs}\in \ZZ^r$ with $|{\sfs}|=\sum_{i=1}^rs_i$. Its support $\supp f$ is defined as
\[
\supp f = \{ \sfs \in \ZZ^r \ |\ a_{\sfs} \not=0\}.
\]
For each $m\in \ZZ$, let $\scL_m^R$ denote the set of  all formal Laurent series $f(t)$ of degree $m$ such that there exists $\sfm \in \ZZ^r$ such that $\sfm + \supp f$ is contained in the cone $C \subset \ZZ^r$ defined by $s_1\geq0, \; s_1+s_2\geq 0, \;\cdots, \; s_1+\cdots + s_{r} \geq 0$. The direct sum $\scL^R:=\oplus_{m\in \ZZ}\scL_m^R$ is a graded ring in an obvious manner. For each $i=1,\ldots,r,$ let $\scL^{R,i}$ denote the subring of $\scL^R$ consisting of series that do not contain negative powers of $t_1,\ldots,t_{i-1}.$ In particular we have $\scL^{R,1}=\scL^R$. For each $m\in \ZZ$, let $R[[t_1,\ldots,t_r]]_m$ denote the set of formal power series in $t_1,\ldots,t_r$ of homogeneous degree $m$.   We define  the ring $R[[t_1,\ldots,t_r]]_{\gr}$ of graded formal power series to be $\oplus_{m\in \ZZ}R[[t_1,\ldots,t_r]]_m$. Note that $\scL^{R,i}$ is a graded $R[[t_1,\ldots,t_r]]_{\gr}$-module. 

Let us apply the above notation to $R=\CK^*(\mathrm{OG}(E))$.  We define an $R[[t_1,\ldots,t_{i-1}]]_{\gr}$-module structure on $\CK^*(\mathbb{P}(F^{\lambda_{i-1}}/D_{i-2}))$ by 
\[
f(t_1,\ldots,t_{i-1})\alpha:=f(\tau_1,\ldots,\tau_{i-1})\alpha,
\]
for each $f(t_1,\ldots,t_{i-1})\in R[[t_1,\ldots,t_{i-1}]]_{\gr}$ and  $\alpha\in \CK^*(\mathbb{P}(F^{\lambda_{i-1}}/D_{i-2}))$ where $\tau_i = c_1((D_i/D_{i-1})^{\vee})$ as before. We can uniquely define a homomorphism $\phi_i:\scL^{R,i} \to \CK^*(\mathbb{P}(F^{\lambda_{i-1}}/D_{i-2}))$ of  graded $R[[t_1,\ldots,t_{i-1}]]_{\gr}$-modules by setting 
\[
t_{i}^{s_i}\cdots t_r^{s_r}\mapsto  \scS_{s_i}((U^\perp -E/F^{\lambda_i})^\vee) \cdots  \scS_{s_r}((U^\perp -E/F^{\lambda_r})^\vee),
\]
for each $s_i,\dots, s_r \in \ZZ$. Note that for $m\in \ZZ$ and $j\geq i$, we have 
\begin{equation}\label{remP}
\phi_i\left(\frac{t_j^{m}}{2+\beta t_j}\right) =\scP_m^{(\lambda_j)}.
\end{equation}
\begin{thm}\label{mainPF} 
Let $\lambda \in \SP(n)$ of length $r$. In $\CK^*(\OG(E))\otimes_{\ZZ}\ZZ[1/2]$, we have
\begin{eqnarray*}
[X_{\lambda}]
%&=& \pi_* \left(\prod_{i=1}^r \frac{c_{n-i+1}((D_i/D_{i-1})^{\vee}\otimes D_{i-1}^{\perp}/U^{\perp})}{(2+\beta \tau_i)}\right)\\
&=& \phi_1\left(\prod_{i=1}^r\frac{t_i^{\lambda_i}}{2+\beta t_i} \prod_{1\leq i<j\leq r } \left(\frac{1-\bar t_i/\bar t_j}{1+\bar t_i/ t_j}\right)\right),
\end{eqnarray*}
where $\bar t = \frac{-t}{1+\beta t}$ is the notation for the formal inverse as before.
\end{thm}
\begin{proof}
As in \cite[Proposition 1]{HIMN}, the definition (\ref{segre vir}) and the property (\ref{push of tensor}) of the relative Segre classes imply that, for $\pi_i : \PP(F^{\lambda_i}/D_{i-1}) \to \PP(F^{\lambda_{i-1}}/D_{i-2})$, we have
\[
\pi_{i*}(\tau_i^s c_{n-i+1}((D_i/D_{i-1})^{\vee} \otimes D_{i-1}^{\perp}/U^{\perp})) = \sum_{p=0}^{\infty} c_p(D_{i-1} - D_{i-1}^{\vee}) \sum_{q=0}^p \binom{p}{q}\beta^{q}
\scS_{\lambda_i+s - p + q}((U^{\perp} - E/F^{\lambda_i})^{\vee})
\]
for each integer $s\geq 0$. Therefore by the same computation used in the proof of \cite[Proposition 3]{HIMN}, we obtain
\begin{equation}\label{lempiphi} 
\pi_{i*}\left(\frac{c_{n-i+1}((D_i/D_{i-1})^{\vee} \otimes D_{i-1}^{\perp}/U^{\perp})}{2+\beta \tau_i}\right)  = \phi_i\left( \frac{t_i^{\lambda_i}}{2+\beta t_i}\prod_{j=1}^{i-1}\frac{ 1-\bar t_j/\bar t_i}{1 -  t_j/\bar t_i}\right).
\end{equation}
Now the claim follows by repeatedly applying (\ref{lempiphi}) to Corollary \ref{corXlambda}. 
\end{proof}
%%%%%%%%%%%%%%%%%%%%%%%%%%%%%%%%%%
%%%%%%%%%%%%%%%%%%%%%%%%%%%%%%%%%%
\subsection{Main theorem}
%%%%%%%%%%%%%%%%%%%%%%%%%%%%%%%%%%
%%%%%%%%%%%%%%%%%%%%%%%%%%%%%%%%%%
Let $\lambda \in \SP(n)$ of length $r$. Let $2m$ be the smallest even integer such that $r \leq 2m$. If $2m>r$, then we set $\lambda_{r+1}=0$. For each $i,j$ such that $1\leq i<j \leq 2m$, we expand the following rational function in $\calL^R$:
\begin{equation}
\left(1 + \beta \bar t_i\right)^{2m-i-1}  \left(1 + \beta \bar t_j\right)^{2m-j} \frac{1 - \bar  t_i/\bar t_j}{1 - t_i/\bar  t_j } = \sum_{a,b \in \ZZ \atop{a \geq  0, a+b \geq 0}} \gamma_{a,b}^{ij} t_i^{a} t_j^{b}.
\end{equation}
Note that $\gamma_{ab}^{ij}\in \ZZ[\beta]$. Note that if $j=2m$, $\gamma^{ij}_{ab}=0$ for all $b>0$. Let $\scP_{m}^{(0)} := (-\beta)^{-m}$ for all $m \in \ZZ_{\leq 0}$. If $A$ is a skewsymmetric $2m \times 2m$ matrix, we denote the Pfaffian of $A$ by $\Pf(A)$. 
We are now ready to state and prove our main result.
%%%%%%%%%%%%%%%%%%%%%%%%%%%%
\begin{thm}\label{mainthm}
The fundamental class of the degeneracy loci $X_{\lambda}$ is given by 
\begin{equation}\label{PfB}
[X_{\lambda}] = \Pf 
\left( 
\sum_{a,b \in \ZZ\atop{a \geq  0, a+b \geq 0}}\gamma_{ab}^{ij}\scP_{\lambda_i+a}^{(\lambda_i)} \scP_{\lambda_j+b}^{(\lambda_j)}
\right)_{1 \leq i < j \leq 2m}.
\end{equation}
\end{thm}
%%%%%%%%%%%%%%%%%%%%%%%%%%%%
\begin{proof}
If $r$ is even, the equality follows from the identity of Schur-Pfaffian
\begin{align}
&\prod_{i=1}^{2m}\frac{t_i^{\lambda_i}}{2+\beta t_i} \prod_{1\leq i<j\leq {2m} } \frac{1-\bar t_i/\bar t_j}{1+\bar t_i/ t_j} \notag\\
&=\Pf \left( \frac{t_i^{\lambda_i}  t_j^{\lambda_j}  \left(1 + \beta \bar t_i\right)^{2m-i-1}  \left(1 + \beta \bar t_j\right)^{2m-j} }{(2+\beta t_i)(2+\beta t_j)} \frac{1 - \bar  t_i/\bar t_j}{1 - t_i/\bar  t_j }\right)_{1 \leq i < j \leq 2m},  \label{pfid}
\end{align}
which is similar to \cite[Lemma 14]{HIMN}. Indeed, we can write
\[
[X_{\lambda}] = \Pf  \left( \phi_1\left(\frac{t_i^{\lambda_i}  t_j^{\lambda_j}  \left(1 + \beta \bar t_i\right)^{2m-i-1}  \left(1 + \beta \bar t_j\right)^{2m-j} }{(2+\beta t_i)(2+\beta t_j)} \frac{1 - \bar  t_i/\bar t_j}{1 - t_i/\bar  t_j }\right)_{1 \leq i < j \leq 2m}\right).
\]
Now (\ref{remP}) implies the claim.

To see the case when $r$ is odd, we add one more stage to the projective tower, namely $\pi_{r+1}: \PP(F^0/D_r) \to \PP(F^{\lambda_r}/D_{r-1})$. If we applying (\ref{push of tensor}), by a direct computation, we obtain (\textit{cf.} \cite[Section 5.3]{HIMN})
\[
(\pi_{r+1})_*(c_{n-r}((D_{r+1}/D_{r})^{\vee}\otimes D_{r}^{\perp}/U^{\perp})) = 1,
\]
where $D_{r+1}/D_r$ is the tautological line bundle of $\PP(F^0/D_r)$. Therefore
\[
[X_{\lambda}] = \pi_*\circ(\pi_{r+1})_* \left(c_{n-r}((D_{r+1}/D_{r})^{\vee}\otimes D_{r}^{\perp}/U^{\perp})\prod_{i=1}^r\frac{ c_{n-i+1}((D_i/D_{i-1})^{\vee}\otimes D_{i-1}^{\perp}/U^{\perp})}{2 + \beta\tau_i}\right).
\]
Thus by noting that
\begin{equation}
(\pi_{r+1})_*(c_{n-r}((D_{r+1}/D_{r})^{\vee}\otimes D_{r}^{\perp}/U^{\perp})) = \phi_{r+1} \left(\prod_{i=1}^{r} \left(\frac{1-\bar t_i/\bar t_{r+1}}{1+\bar t_i/ t_{r+1}}\right)\right)
\end{equation}
we have 
\begin{eqnarray*}
[X_{\lambda}] = \phi_1\left(\prod_{i=1}^{r}\frac{t_i^{\lambda_i}}{2+\beta t_i} \prod_{1\leq i<j\leq r+1 } \left(\frac{1-\bar t_i/\bar t_j}{1+\bar t_i/ t_j}\right)\right).
\end{eqnarray*}
Therefore the claim follows from the next identity similar to (\ref{pfid})
\begin{align}
&\prod_{i=1}^{r}\frac{t_i^{\lambda_i}}{2+\beta t_i} \prod_{1\leq i<j\leq {r+1} } \frac{1-\bar t_i/\bar t_j}{1+\bar t_i/ t_j} \notag\\
&=\Pf \left( \frac{t_i^{\lambda_i}  t_j^{\lambda_j}  \left(1 + \beta \bar t_i\right)^{2m-i-1}  \left(1 + \beta \bar t_j\right)^{2m-j} }{[2]_i[2]_j} \frac{1 - \bar  t_i/\bar t_j}{1 - t_i/\bar  t_j }\right)_{1 \leq i < j \leq 2m},  \label{pfid}
\end{align}
where $[2]_i=\begin{cases}
2+\beta t_i & i=1,\dots, r\\
1 & i=r+1
\end{cases}$.

%we define a graded $R[[ t_1,\dots, t_{i-1}]]_{\gr}$-module homomorphism for each $i\geq 1$, 
%\[
%p_i: R[[ t_1,\dots, t_i]]_{\gr} \to R[[ t_1,\dots,  t_{i-1}]]_{\gr}
%\]
%by 
%\begin{equation}
%p_i(t_i^s)=\sum_{p=0}^{\infty} \left(\sum_{l=0}^p(-1)^le_{p-l}(\bar t_1,\dots,\bar t_{i-1})h_l( t_1,\dots, t_{i-1})\right)\sum_{q=0}^p\binom{p}{q} \beta^q \scS_{\lambda_i+s-p+q}((U^{\perp} - E/F^{\lambda_i})^{\vee} ) 
%\end{equation}
%for all $s \geq 0$. Note that, for each $i=1,\dots, r$, we have the following commutative diagram of $R[[ t_1,\dots,  t_{i-1}]]_{\gr}$-modules
%\[
%\xymatrix{
%R[[ t_1,\dots, t_i]]_{\gr} \ar[r]_{p_i}\ar[d]& R[[ t_1,\dots,  t_{i-1}]]_{\gr}\ar[d]\\
%\CK^*(\PP(F^{\lambda_i}/D_{i-1})) \ar[r]_{\pi_{i*}} & \CK^*(\PP(F^{\lambda_{i-1}}/D_{i-2}))
%}
%\]
%where the vertical maps sends $1$ to $c_{n-i+1}((D_i/D_{i-1})^{\vee} \otimes D_{i-1}^{\perp}/U^{\perp})$ and $t_j$ to $\tau_j$. With this, we have
%\begin{eqnarray*}
%[X_{\lambda}] 
%&=& \pi_* \left(\prod_{i=1}^r\frac{ c_{n-i+1}((D_i/D_{i-1})^{\vee}\otimes D_{i-1}^{\perp}/U^{\perp})}{2 + \beta\tau_i}\right)= p_1 \cdots p_r\left(\prod_{i=1}^r\frac{1}{1+\beta t_i}\right).
%\end{eqnarray*}
%As shown in \cite[Section 5.3]{HIMN}, we can show that $p_m(1)=1$ for $m>r$. Thus
%\begin{eqnarray*}
%[X_{\lambda}] &=& p_1\cdots p_rp_{r+1}\left(\prod_{i=1}^r\frac{1}{2+\beta t_i}\right)
%= \phi_1\left(\prod_{i=1}^{r}\frac{t_i^{\lambda_i}}{2+\beta t_i} \prod_{1\leq i<j\leq r+1 } \left(\frac{1-\bar t_i/\bar t_j}{1+\bar t_i/ t_j}\right)\right).
%\end{eqnarray*}
%Therefore the claim follows from the identity similar to (\ref{pfid}).
\end{proof}
%%%%%%%%%%%%%%%%%%%%%%%%%%%%%%%%%%
%%%%%%%%%%%%%%%%%%%%%%%%%%%%%%%%%%
\section{Appendix: $K$-theory of odd quadric bundles $Q(E)$}\label{appendix}\label{app}
%%%%%%%%%%%%%%%%%%%%%%%%%%%%%%%%%%
%%%%%%%%%%%%%%%%%%%%%%%%%%%%%%%%%%
In this section, we show Identity (\ref{rel1}) in the $K$-theory of the odd quadric bundle, which we used in the main body of this paper. As an application, we also exhibit a presentation of the $K$-theory of the quadric bundle. 

Let $E$ be a vector bundle of rank $2n+1$ over a smooth quasiprojective variety $X$ with a symmetric non-degenerate bilinear form with values in a line bundle $L$. Let $S$ be the tautological line bundle of $\PP(E)$.  The quadric bundle $Q(E) \subset \PP(E)$ is given by
\[
Q(E) = \{(\ell,x) \in \PP(E)|\ \ell \in \PP(E_x), \ell \mbox{ is isotropic }\}.
\] 
%In $\CK^*(\PP(E))$, we have
%\[
%[Q(E)] = c_1(S^{\vee}\otimes S^{\vee} \otimes L).
%\]
Let $U$ be a maximal isotropic subbundle of $E$. We have $U^{\perp}/U \otimes U^{\perp}/U =L$. Consider the following diagram of obvious inclusions.
\[ 
\xymatrix{
Q(E) \ar[rrr] &&& \PP(E) \\
Q(E) \cap \PP(U^{\perp}) \ar[rrr]\ar[u]_{\iota'}&&&\PP(U^{\perp}) \ar[u]\\
\PP(U)\ar@/_1pc/[urrr]\ar[u]_{\iota}&&&
}
\]
%%%%%%%%%%%%%%%%%%%%%%%%%%%%%%%%%%
\begin{lem}\label{cor1}
In $\CK^*(Q(E))$, we have
\begin{equation}\label{rel1}
(2+\beta c _1(S^{\vee}\otimes U^{\perp}/ U)) [\PP(U)] = c_n(S^{\vee}\otimes E/U^{\perp}).
\end{equation}
\end{lem}
%%%%%%%%%%%%%%%%%%%%%%%%%%%%%%%%%%
\begin{proof}
We show the identity by computing the class $[Q(E) \cap \PP(U^{\perp})]$ in $\CK^*(Q(E))$ in two different ways. $\PP(U)$ is a divisor in $\PP(U^{\perp})$ and the corresponding line bundle over $\PP(U^{\perp})$ is $S^{\vee}\otimes U^{\perp}/ U$. The class $[\PP(U)]$ in $\CK^*(\PP(U^{\perp}))$ is equal to $c _1(S^{\vee}\otimes U^{\perp}/ U)$. The scheme theoretic intersection $Q(E)\cap \PP(U^\perp)$ is not reduced and defines the Weil divisor $2 \PP(U)$ on $\PP(U^{\perp})$, which is obviously a strict normal crossing.
%The Weil divisor $2\PP(U)$ %is a strict normal crossing divisor (\cite[Definition 3.1.4]{LevineMorel}) and it 
%defines the subscheme $Q(E) \cap \PP(U^{\perp})$. 
Thus, following \cite[Section 7.2.1]{LevineMorel}, we can compute the fundamental class $1_{Q(E) \cap \PP(U^{\perp})}$ in $\CK^*({Q(E) \cap \PP(U^{\perp})})$ as
\begin{equation}
1_{Q(E) \cap \PP(U^{\perp})} = \iota_*(2 + \beta c _1(S^{\vee}\otimes U^{\perp}/U)).
\end{equation}
In $\CK^*(Q(E))$, the fundamental class $[Q(E) \cap \PP(U^{\perp})]$ is defined as the pushforward $\iota_*'1_{Q(E) \cap \PP(U^{\perp})}$. Since the class $2 + \beta c _1(S^{\vee}\otimes U^{\perp}/U)$ pulls back from $Q(E)$, the projection formula applied to $\iota'\circ \iota$ implies 
\begin{eqnarray*}
[Q(E) \cap \PP(U^{\perp})]
%&=& \iota_*'\iota_*(2 + \beta c _1(S^{\vee}\otimes U^{\perp}/U))\\
&=& (2 + \beta c _1(S^{\vee}\otimes U^{\perp}/U))\cdot [\PP(U)].
\end{eqnarray*}
On the other hand, the scheme $Q(E)\cap \PP(U^{\perp})$ is the locus where the obvious bundle map $S \to E/U^{\perp}$ has rank zero, and its codimension in $Q(E)$ is $n$. Thus, by \cite[Lemma 6.6.7]{LevineMorel}, we have
\[
[Q(E)\cap \PP(U^{\perp})]=c_n(S^{\vee}\otimes E/U^{\perp})
\]
in $\CK^*(Q(E))$. 
\end{proof}
%%%%%%%%%%%%%%%%%%%%%%%%%%%%%%%%%%
\begin{rem}
It is easy to generalize this to the algebraic cobordism of $Q(E)$. Indeed, $2+\beta c _1(S^{\vee}\otimes U^{\perp}/ U)$ in the above formula is nothing but $F_1^{(2)}(c _1(S^{\vee}\otimes U^{\perp}/ U))$ in the notation of \cite[Section 3.1.2]{LevineMorel}.
\end{rem}
%%%%%%%%%%%%%%%%%%%%%%%%%%%%%%%%%%
\begin{thm}[\textit{cf.} \cite{BuchSamuel}]
We have
\[
\CK^*(Q(E)) \cong \CK^*(X)[h,f]/I,
\] 
where the ideal $I$ is generated by the relations (\ref{rel1}) and 
\begin{equation}\label{rel2}
f^2 = c_n(S^{\vee} \otimes E/U-S^{\vee}\otimes S^{\vee}\otimes L))f.
\end{equation}
\end{thm}
%%%%%%%%%%%%%%%%%%%%%%%%%%%%%%%%%%
\begin{proof}
It is known that $\CK^*(Q(E))$ is a free module over $\CK^*(X)$ with basis
\[
1,h,\dots,h^{n-1},   \ \ \ f,hf,\dots,fh^{n-1}, 
\]
where $f=[\PP(U)]$ and $h=c_1(S^{\vee})$. For example, it follows from the standard fact that $Q(E)$ admits a cell decomposition of $Q(E)$, combined with the cellular decomposition property \cite[Section 5.1.2]{LevineMorel}. 
%%%%%%%%%%%%%%%%%%%%%%%%%%%%%%%%%%
The relation (\ref{rel2}) is identical to the one in \cite[Appendix (A.4)]{AndersonFulton2} for cohomology case and it also holds in $\CK^*(Q(E))$. Indeed, the self-intersection formula $f^2=\iota'_*\iota_* c_n(N_{\PP(U)}Q(E))$ holds also in connective $K$-theory (\cite{LevineMorel}) and the normal bundle $N_{\PP(U)}Q(E)$ of $\PP(U)$ in $Q(E)$ sits in the short exact sequence 
\[
0 \to N_{\PP(U)}Q(E) \to S^{\vee} \otimes E/U\to S^{\vee}\otimes S^{\vee}\otimes L \to 0.
\]
%Equation (\ref{rel1}) becomes
%\begin{equation}\label{eq2(3)}
%(2+\beta (h+\kappa + \beta h \kappa))f = \sum_{p=0}^n c_p(E/U^{\perp}) \sum_{q=0}^p \beta^q \binom{p}{q} h^{n-p+q}.
%\end{equation}
%where we set $\kappa=c_1(U^{\perp}/U)$. 
We find that (\ref{rel1}) is a polynomial equation in $h, f$ with coefficients in $\CK^*(X)$.  The top degree of $h$ in (\ref{rel1}) is $n$ and its coefficient is $c(E/U^{\perp};\beta)$ which is invertible in $\CK^*(X)$. Therefore, as we know an additive basis, the relations determine the ring structure.
\end{proof}
%%%%%%%%%%%%%%%%%%%%%%%%%%%%%%%%%%
%%%%%%%%%%%%%%%%%%%%%%%%%%%%%%%%%%
%%%%%%%%%%%%%%%%%%%%%%%%%%%%%%%%%%
%%%%%%%%%%%%%%%%%%%%%%%%%%%%%%%%%%

\textbf{Acknowledgements.} 
%A considerable part of this work developed while the first and third authors were affiliated to KAIST, which they would like to thank for the excellent working conditions.
Part of this work developed while the first author was affiliated to  POSTECH, which he would like to thank for the excellent working conditions. He would also like to gratefully acknowledge the support of the National Research Foundation of Korea (NRF) through the grants funded by the Korea government (MSIP) (2014-001824 and 2011-0030044).
The second author  is supported by Grant-in-Aid for Scientific Research (C) 24540032, 15K04832.
The fourth author is supported by Grant-in-Aid for Scientific Research (C) 25400041.

\bibliography{references}{}
\bibliographystyle{acm}

\begin{small}
{\scshape
\noindent Thomas Hudson, Department of Mathematical Sciences, KAIST, 291 Daehak-ro, Yuseong-gu, Daejeon, 34141, Republic of Korea (South)
}
\end{small}

{\textit{email address}: \tt{hudson.t@kaist.ac.kr}}

\

\begin{small}
{\scshape
\noindent  Takeshi Ikeda, Department of Applied Mathematics, Okayama University of Science, Okayama 700-0005, Japan
}
\end{small}

{\textit{email address}: \tt{ike@xmath.ous.ac.jp}}

\

\begin{small}
{\scshape
\noindent Tomoo Matsumura, Department of Applied Mathematics, Okayama University of Science, Okayama 700-0005, Japan
}
\end{small}

{\textit{email address}: \tt{matsumur@xmath.ous.ac.jp}}

\

\begin{small}
{\scshape
\noindent Hiroshi Naruse, Graduate School of Education, University of Yamanashi, Yamanashi 400-8510, Japan
}
\end{small}

{\textit{email address}: \tt{hnaruse@yamanashi.ac.jp}}

\end{document}